\documentclass{amsart}
\usepackage{bbm}
\usepackage{bbding}
\usepackage[mathcal]{euscript}
\usepackage{graphicx}
\usepackage{amsthm}
\usepackage{calc}
\usepackage{mathrsfs,dsfont}
\usepackage{CJK,fancyhdr,amscd}
\usepackage{anysize}
\usepackage{amsmath,amssymb,amsfonts}
\usepackage{epsfig,enumerate}
\usepackage{latexsym}
\usepackage{indentfirst, latexsym}
\usepackage{graphics}
\usepackage[all,poly,knot]{xy}
\usepackage[colorlinks,linkcolor=blue,anchorcolor=blue,citecolor=blue,pagebackref]{hyperref}
\usepackage{geometry}
\numberwithin{equation}{section}
\newtheorem{theorem} {Theorem} [section]
\newtheorem{proposition}[theorem]{Proposition}
\newtheorem{corollary}  [theorem]     {Corollary}
\newtheorem{lemma}  [theorem]     {Lemma}

\newtheorem{conjecture}  [theorem]     {Conjecture}

\theoremstyle{definition}

\renewcommand*\backref[1]{}
\renewcommand*\backrefalt[4]{ \ifcase #1 \or (cited on page #2) \else (cited on pages #2) \fi}

\begin{document}

\title{Hermitian Connections with Parallel Torsion and the Fino--Vezzoni Conjecture}

\author{Shuwen Chen}
\address{Shuwen Chen. School of Mathematical Sciences, Chongqing Normal University, Chongqing 401331, China}
\email{{3153017458@qq.com}}\thanks{Zheng is the corresponding author. He is partially supported by National Natural Science Foundations of China with the grant No. 12471039 and  12141101, and is supported by the 111 Project D21024.}

\author{Fangyang Zheng}
\address{Fangyang Zheng. School of Mathematical Sciences, Chongqing Normal University, Chongqing 401331, China}
\email{20190045@cqnu.edu.cn; franciszheng@yahoo.com} \thanks{}


\begin{abstract}
We prove that the Fino--Vezzoni conjecture holds on every compact complex manifold carrying a Hermitian connection with parallel torsion. More precisely, if a compact connected complex manifold carries a Hermitian metric $g$ and a Hermitian connection $D$ with $DT^D=0$, then the coexistence of a balanced metric and a pluriclosed metric forces the existence of a $D$-parallel K\"ahler metric. The proof combines a holonomy symmetrization onto $D$-parallel forms with a finite-dimensional volume-maximization argument. In particular, the conjecture holds on every compact Bismut torsion-parallel manifold. In the non-balanced case, the obstruction of Zhao--Zheng yields the stronger conclusion that balanced and pluriclosed metrics cannot coexist.
\end{abstract}

\maketitle

\markleft{Shuwen Chen and Fangyang Zheng}
\markright{Hermitian connections with parallel torsion and the Fino--Vezzoni conjecture}

\tableofcontents

\section{Introduction}

An important problem in non-K\"ahler geometry is to understand whether two special Hermitian structures of quite different nature can occur on the same compact complex manifold. Let $(M^n,g)$ be a Hermitian manifold of complex dimension $n$, and let $\omega$ be the fundamental form of $g$. Following the notation used in \cite{ZhaoZBTP}, Gauduchon's torsion $1$-form $\eta$ is the global $(1,0)$-form determined by $\partial\omega^{n-1}=-\eta\wedge\omega^{n-1}$. The metric $g$ is called {\em balanced} when $\eta=0$, or equivalently, when $d\omega^{n-1}=0$. It is called {\em pluriclosed}, also known as SKT, when $\partial\overline{\partial}\omega=0$. Both conditions are natural weakenings of the K\"ahler condition, but neither one implies the other in general.

Fino and Vezzoni proposed the following conjecture; see \cite{FV2,FV16}.

\begin{conjecture}[{\bf Fino--Vezzoni}]\label{conj:FV}
Let $M^n$ be a compact complex manifold. If $M$ admits a balanced Hermitian metric and a pluriclosed Hermitian metric, then $M$ admits a K\"ahler metric.
\end{conjecture}

The two metrics in Conjecture \ref{conj:FV} need not be the same. This distinction is essential. Indeed, if a single Hermitian metric is both balanced and pluriclosed, then it is K\"ahler by a standard result of Alexandrov--Ivanov \cite[Remark~1]{AI}; see also Lemma \ref{lem:balanced-pluriclosed-kahler} below. Also, the interesting dimensions start from $n=3$, since in complex dimension two the balanced condition already reads $d\omega=0$.

A number of special cases of Conjecture \ref{conj:FV} are known. Chiose \cite{Chiose} proved the conjecture for manifolds in Fujiki class $\mathcal C$. In fact, he showed the stronger statement that a manifold in class $\mathcal C$ carrying a pluriclosed metric is K\"ahler. Verbitsky \cite{Verbitsky} verified the conjecture for twistor spaces by showing that a twistor space admitting a pluriclosed metric must be K\"ahler. Popovici \cite{PopStar} obtained an affirmative answer under additional assumptions involving pluriclosed star split metrics. Fu-Li-Yau \cite{FuLiYau} proved that the conjecture holds for a special type of non-K\"ahler Calabi-Yau threefolds. Also, Fei \cite{Fei} confirmed the conjecture for generalized Calabi-Gray manifolds. On the other hand, Otiman \cite{Otiman} proved that Oeljeklaus--Toma manifolds admit no balanced metrics, while Angella--Otiman \cite{AngellaOtiman} established analogous nonexistence results for compact Vaisman manifolds. Thus the Fino--Vezzoni conjecture also holds on these classes.

There has also been substantial progress for compact quotients of Lie groups with invariant complex structures. The original work of Fino--Vezzoni verified the conjecture for nilmanifolds of dimensions six and eight, as well as for six-dimensional solvmanifolds with holomorphically trivial canonical bundle \cite{FV2}. They subsequently extended the result to complex nilmanifolds in arbitrary dimension \cite{FV16}. The relevant structure theory was later completed by Arroyo--Nicolini \cite{ArroyoNicolini}, who proved that every nilmanifold admitting an invariant SKT structure is either a torus or two-step nilpotent. Semisimple cases were studied in \cite{FGV,Podesta,GiustiPodesta}. More recently, Kwong \cite{KwongSemi} showed that compact quotients of real semisimple Lie groups endowed with regular complex structures cannot simultaneously admit a balanced metric and a pluriclosed metric. Fino--Paradiso established the conjecture for almost abelian solvmanifolds and obtained further results for almost nilpotent solvmanifolds \cite{FinoParadisoAA,FinoParadisoAN,FinoParadisoSix}. Freibert--Swann \cite{FreibertSwann} verified the conjecture for compact $2$-step solvmanifolds whose invariant complex structure is of pure type, as well as in real dimension six; Fusi--Gentili have recently removed these restrictions for unimodular $2$-step solvable Lie algebras \cite{FusiGentili}. For Lie algebras with a codimension-two abelian ideal, Li--Zheng handled the $J$-invariant case \cite{LiZhengFV}, and Cao--Zheng subsequently removed the $J$-invariance assumption \cite{CaoZheng}. Recent developments also include a pluriclosed-flow result of Fino--Vezzoni \cite{FinoVezzoniFlow} and Kwong's proof for compact discrete quotients of complex homogeneous spaces with compact isotropy \cite{Kwong}. We refer to \cite{ZhengSurvey} for a recent survey of the conjecture and related questions.

We now turn to Bismut torsion-parallel manifolds. The Bismut connection $\nabla^b$ is the unique Hermitian connection with totally skew-symmetric torsion \cite{Bismut}. A Hermitian metric is called {\em Bismut torsion parallel}, or {\em BTP}, if $\nabla^bT^b=0$. The Bismut K\"ahler-like case already provides a compatibility obstruction \cite{ZhaoZ19Str}. Zhao--Zheng proved the Fino--Vezzoni conjecture for compact non-balanced BTP threefolds \cite{ZhaoZBTP}. Subsequently, by studying compact balanced BTP threefolds, they established the conjecture in the balanced case as well, and hence for all compact BTP threefolds \cite{ZhaoZ25}. More recently, Chen--Zheng extended the balanced case to any complex dimension, proving the conjecture for compact balanced BTP manifolds \cite{ChenZhengFV}.

The result of this paper is naturally formulated at the level of a general Hermitian connection with parallel torsion. Here a Hermitian connection means a connection compatible with both the Hermitian metric and the {almost} complex structure.

\begin{theorem}\label{thm:parallel-main}
Let $(M^n,g)$ be a compact connected Hermitian manifold with a Hermitian connection $D$ whose torsion is $D$-parallel (namely, $DT^D=0$). If $M$ admits both a balanced Hermitian metric and a pluriclosed Hermitian metric, then $M$ admits a K\"ahler metric. Moreover, the K\"ahler metric can be chosen to be $D$-parallel.
\end{theorem}

Taking $D$ to be the Bismut connection gives the following consequence.

\begin{corollary}\label{thm:main}
The Fino--Vezzoni conjecture holds on every compact BTP manifold. Moreover, if $(M^n,g)$ is a compact non-balanced BTP manifold, then the underlying complex manifold cannot admit a balanced Hermitian metric and a pluriclosed Hermitian metric simultaneously.
\end{corollary}

The proof uses two features of a Hermitian connection with $DT^D=0$. First, exterior differentiation preserves $D$-parallel forms. Second, unitary holonomy provides an averaging procedure that preserves positivity. These observations reduce the problem to a finite-dimensional affine space of parallel $(1,1)$-forms, where the balanced condition is detected by the first variation of the volume. The determinant-maximization step is related to the homogeneous-space argument in \cite{Kwong}, but no transitive group action is required here.

\section{Preliminaries}\label{PRE}

Let $(M^n,J,g)$ be a Hermitian manifold and let $\omega$ be its fundamental form, $\omega(X,Y)=g(JX,Y)$. A connection $D$ on $TM$ is called {\em Hermitian} if it is compatible with both $g$ and $J$, namely
$$
Dg=0,\qquad DJ=0.
$$
Its torsion tensor is
$$
T^D(X,Y)=D_XY-D_YX-[X,Y].
$$
We write $dV_g$ for the Riemannian volume form determined by $g$ and the complex orientation. The Hodge star $*_g$ is characterized, for real $k$-forms $\alpha$ and $\beta$, by
$$
\alpha\wedge *_g\beta=\langle\alpha,\beta\rangle_g\,dV_g,
$$
where $\langle\ ,\ \rangle_g$ denotes the inner product on differential forms induced by $g$. We extend $*_g$ complex linearly to complex-valued forms. Since a Hermitian connection preserves $g$ and $J$, its parallel transports are unitary and preserve the complex orientation. Hence $dV_g$ and $*_g$ are $D$-parallel; in particular,
$$
D_X(*_g\alpha)=*_g(D_X\alpha).
$$

The Bismut connection $\nabla^b$ is the unique Hermitian connection whose torsion is totally skew-symmetric \cite{Bismut}. We denote its torsion by $T^b$. Following Zhao--Zheng \cite{ZhaoZBTP}, a Hermitian metric is called {\em Bismut torsion parallel}, or {\em BTP}, if
$$
\nabla^bT^b=0.
$$

We shall use the following standard result of Alexandrov--Ivanov \cite[Remark~1]{AI}.

\begin{lemma}[Alexandrov--Ivanov \cite{AI}]\label{lem:balanced-pluriclosed-kahler}
Let $(M^n,\omega)$ be a Hermitian manifold. If $d\omega^{n-1}=0$ and $\partial\overline\partial\omega=0$, then $d\omega=0$.
\end{lemma}

Recall that a Hermitian metric is called {\em strongly Gauduchon} if there exists an $(n,n-2)$-form $\Omega$ satisfying $\partial\omega^{n-1}=\overline\partial\Omega$; see \cite{Pop}. We shall also use the following consequence of the BTP structure theory \cite[Corollary~1.11]{ZhaoZBTP}.

\begin{lemma}[Zhao--Zheng \cite{ZhaoZBTP}]\label{lem:ZZ-obstruction}
Let $(M^n,g)$ be a compact non-balanced BTP manifold. Then $g$ is not strongly Gauduchon and $H^{0,1}_{\overline\partial}(M)\neq0$. Moreover, $M$ does not satisfy the $\partial\overline\partial$-lemma. In particular, $M$ does not admit any K\"ahler metric.
\end{lemma}

\section{Parallel-torsion symmetrization}\label{SYM}

Every Hermitian metric on a complex curve is K\"ahler. We therefore assume $n\geq2$ in the nontrivial part of the argument.

Fix a compact connected Hermitian manifold $(M^n,g)$ and a Hermitian connection $D$ satisfying $DT^D=0$. For $0\leq k\leq2n$, set
$$
\mathcal E_D^k=\{\alpha\in\Omega^k(M):D\alpha=0\}.
$$
We first regard $\mathcal E_D^k$ as the space of real $D$-parallel $k$-forms and use the same notation for its complexification. Since a parallel form is determined by its value at one point, each $\mathcal E_D^k$ is finite-dimensional.

\begin{lemma}\label{lem:d-parallel}
If $DT^D=0$, then $d\mathcal E_D^k\subseteq\mathcal E_D^{k+1}$ for every $k$. Consequently, $(\mathcal E_D^\bullet,d)$ is a differential subcomplex of the de Rham complex.
\end{lemma}

\begin{proof}
Let $\alpha\in\mathcal E_D^k$. For vector fields $X_0,\ldots,X_k$,
\begin{align*}
d\alpha(X_0,\ldots,X_k)
={}&\sum_{i=0}^k(-1)^i(D_{X_i}\alpha)(X_0,\ldots,\widehat X_i,\ldots,X_k)\\
&+\sum_{0\leq i<j\leq k}(-1)^{i+j+1}\alpha\big(T^D(X_i,X_j),X_0,\ldots,\widehat X_i,\ldots,\widehat X_j,\ldots,X_k\big).
\end{align*}
Since $D\alpha=0$, the first sum vanishes, while the second is an algebraic contraction of $T^D$ with $\alpha$. The assumptions $DT^D=0$ and $D\alpha=0$ therefore imply $D(d\alpha)=0$. Hence $d\alpha\in\mathcal E_D^{k+1}$. Since $\mathcal E_D^k\subseteq\Omega^k(M)$ for every $k$ and $d^2=0$ on the de Rham complex, the restriction of $d$ to $\mathcal E_D^\bullet$ also satisfies $d^2=0$. Thus $(\mathcal E_D^\bullet,d)$ is a differential subcomplex of the de Rham complex.
\end{proof}

By Section \ref{PRE}, the Hodge star is $D$-parallel, and hence $*_g\mathcal E_D^k=\mathcal E_D^{2n-k}.$

\begin{lemma}\label{lem:integration-pairing}
For every $k$, the integration pairing
$$
\mathcal E_D^k\times\mathcal E_D^{2n-k}
\longrightarrow\mathbb R,\qquad
(\alpha,\beta)\longmapsto\int_M\alpha\wedge\beta
$$
is nondegenerate. Equivalently, it induces an isomorphism
$$
\mathcal E_D^k\longrightarrow(\mathcal E_D^{2n-k})^*,\qquad
\alpha\longmapsto\left(\beta\longmapsto\int_M\alpha\wedge\beta\right),
$$
where $(\mathcal E_D^{2n-k})^*$ denotes the dual vector space of $\mathcal E_D^{2n-k}$.
\end{lemma}

\begin{proof}
The $D$-parallelness of $*_g$ gives $\dim\mathcal E_D^k=\dim\mathcal E_D^{2n-k}$. If $0\neq\alpha\in\mathcal E_D^k$, then $*_g\alpha\in\mathcal E_D^{2n-k}$ and
$$
\int_M\alpha\wedge *_g\alpha
=
\int_M|\alpha|_g^2\,dV_g>0.
$$
This proves nondegeneracy in the first factor. The same argument, applied to a nonzero element of $\mathcal E_D^{2n-k}$, gives nondegeneracy in the second factor. Since the spaces are finite-dimensional and have the same dimension, the induced map to $(\mathcal E_D^{2n-k})^*$ is an isomorphism.
\end{proof}

This pairing gives a projection onto $D$-parallel forms. The construction is the parallel analogue of the symmetrization in \cite[Proposition~A.1]{Kwong}.

\begin{proposition}\label{prop:symmetrization}
For every $k$, there is a unique linear map $\mathscr S_k:\Omega^k(M)\to\mathcal E_D^k$ characterized by
\begin{equation}\label{eq:sym-characterization}
\int_M\mathscr S_k(\alpha)\wedge\beta=\int_M\alpha\wedge\beta,\qquad \beta\in\mathcal E_D^{2n-k}.
\end{equation}
Moreover, $\mathscr S_k$ is the identity on $\mathcal E_D^k$ and $\mathscr S_{k+1}d=d\mathscr S_k$.
\end{proposition}

\begin{proof}
Fix $\alpha\in\Omega^k(M)$. The map
$$
\beta\longmapsto\int_M\alpha\wedge\beta,\qquad
\beta\in\mathcal E_D^{2n-k},
$$
is an element of $(\mathcal E_D^{2n-k})^*$. By the isomorphism in Lemma \ref{lem:integration-pairing},
there exists a unique $\mathscr S_k(\alpha)\in\mathcal E_D^k$ such that
$$
\int_M\mathscr S_k(\alpha)\wedge\beta
=
\int_M\alpha\wedge\beta
$$
for every $\beta\in\mathcal E_D^{2n-k}$. The uniqueness also shows that the assignment
$\alpha\mapsto\mathscr S_k(\alpha)$ is linear. If $\alpha\in\mathcal E_D^k$, then $\alpha$ itself satisfies
\eqref{eq:sym-characterization}; hence uniqueness gives $\mathscr S_k(\alpha)=\alpha$. In particular, $\mathscr S_k^2=\mathscr S_k$.

It remains to prove compatibility with $d$. Let $\gamma\in\mathcal E_D^{2n-k-1}$. Lemma \ref{lem:d-parallel} gives $d\gamma\in\mathcal E_D^{2n-k}$. By Stokes' theorem and \eqref{eq:sym-characterization},
\begin{align*}
\int_Md\mathscr S_k(\alpha)\wedge\gamma
&=(-1)^{k+1}\int_M\mathscr S_k(\alpha)\wedge d\gamma
=(-1)^{k+1}\int_M\alpha\wedge d\gamma=\int_Md\alpha\wedge\gamma=\int_M\mathscr S_{k+1}(d\alpha)\wedge\gamma.
\end{align*}
By Lemma \ref{lem:d-parallel}, $d\mathscr S_k(\alpha)\in\mathcal E_D^{k+1}$, while $\mathscr S_{k+1}(d\alpha)\in\mathcal E_D^{k+1}$ by definition. Nondegeneracy in Lemma \ref{lem:integration-pairing} therefore gives
$$
d\mathscr S_k(\alpha)=\mathscr S_{k+1}(d\alpha).
$$
This completes the proof of the proposition. \end{proof}

We use the same notation $\mathscr S_k$ for its complex-linear extension to complex-valued $k$-forms, and write $\mathscr S=\bigoplus_k\mathscr S_k$.

The definition of $\mathscr S$ in Proposition \ref{prop:symmetrization} is particularly convenient for proving the identity $\mathscr S d=d\mathscr S$. However, the preservation of positivity is not immediate from the defining relation \eqref{eq:sym-characterization}. For this purpose, we give below a holonomy description of $\mathscr S$.

Fix a point $p\in M$. Let $\operatorname{Hol}_p(D)$ denote the full holonomy group of $D$ at $p$, namely, the group generated by the $D$-parallel transports along piecewise smooth loops based at $p$. Since $D$ is a Hermitian connection, its parallel transports preserve both $g$ and $J$, and hence $\operatorname{Hol}_p(D)\subset U(T_pM,J_p,g_p)$. We denote by
$$
\mathcal H=\overline{\operatorname{Hol}_p(D)}
$$
its closure in $U(T_pM,J_p,g_p)$. Thus $\mathcal H$ is a compact group. Let $d\mu_{\mathcal H}$ be the normalized Haar measure on $\mathcal H$, so that $\mu_{\mathcal H}(\mathcal H)=1$.

For any $x\in M$, choose a piecewise smooth curve $\gamma$ joining $p$ to $x$, and denote by
$$
P_\gamma:T_pM\longrightarrow T_xM
$$
the $D$-parallel transport along $\gamma$. If $\alpha\in\Omega^k(M)$, then $\alpha_x\in\Lambda^kT_x^*M$, and the pull-back $P_\gamma^*\alpha_x$ belongs to $\Lambda^kT_p^*M$. We define
\begin{equation}\label{eq:Reynolds}
\mathcal R_\alpha(x)
=
\int_{\mathcal H}
h^*\bigl(P_\gamma^*\alpha_x\bigr)\,
d\mu_{\mathcal H}(h).
\end{equation}
Here $h^*$ denotes the natural pull-back action of $h\in\mathcal H$ on $\Lambda^kT_p^*M$. By construction, $\mathcal R_\alpha(x)$ belongs to the invariant subspace
$$
(\Lambda^kT_p^*M)^{\mathcal H}
=
\left\{
\xi\in\Lambda^kT_p^*M:
h^*\xi=\xi\ \text{for every }h\in\mathcal H
\right\}.
$$

We first check that \eqref{eq:Reynolds} is independent of the choice of $\gamma$. Suppose that $\gamma'$ is another curve from $p$ to $x$. Then
$$
P_{\gamma'}=P_\gamma\circ a
$$
for some $a\in\operatorname{Hol}_p(D)\subset\mathcal H$. Therefore, using the invariance of the Haar measure,
\begin{align*}
\int_{\mathcal H}
h^*P_{\gamma'}^*\alpha_x\,d\mu_{\mathcal H}(h)
&=
\int_{\mathcal H}
h^*a^*P_\gamma^*\alpha_x\,d\mu_{\mathcal H}(h)=
\int_{\mathcal H}
h^*P_\gamma^*\alpha_x\,d\mu_{\mathcal H}(h).
\end{align*}
Hence $\mathcal R_\alpha(x)$ is well-defined.

We also note that the map $x\mapsto\mathcal R_\alpha(x)$ is smooth. Indeed, for any fixed $x_0\in M$, choose a curve from $p$ to $x_0$. On a sufficiently small neighborhood $U$ of $x_0$, one can choose curves from $x_0$ to $x\in U$ which depend smoothly on $x$. The corresponding parallel transports then depend smoothly on $x$.
Since the definition above is independent of the chosen paths, these local constructions agree on overlaps and give a globally defined smooth map
$$
\mathcal R_\alpha:
M\longrightarrow(\Lambda^kT_p^*M)^{\mathcal H}.
$$
We may therefore average once more over $M$. Put
\begin{equation}\label{eq:holonomy-average}
\alpha_p^{\mathrm{av}}
=
\frac{1}{\operatorname{Vol}_g(M)}
\int_M\mathcal R_\alpha(x)\,dV_g(x),
\end{equation}
where $dV_g$ denotes the volume form of $g$ and $\operatorname{Vol}_g(M)=\int_MdV_g$. Since each $\mathcal R_\alpha(x)$ is fixed by $\mathcal H$, so is $\alpha_p^{\mathrm{av}}$.

By the holonomy principle, an element of $\Lambda^kT_p^*M$ which is fixed by the full holonomy group extends uniquely by $D$-parallel transport to a global $D$-parallel $k$-form; see \cite[Chapter~II]{KN}. Since $\alpha_p^{\mathrm{av}}$ is $\mathcal H$-invariant, and hence $\operatorname{Hol}_p(D)$-invariant, it determines a unique global $D$-parallel form.  We denote this form by
$$
\widetilde{\mathscr S}_k(\alpha)\in\mathcal E_D^k.
$$
Thus the above construction defines a linear map
$$
\widetilde{\mathscr S}_k:
\Omega^k(M)\longrightarrow\mathcal E_D^k.
$$
The next proposition shows that this holonomy-average operator is precisely the symmetrization operator $\mathscr S_k$ defined in Proposition \ref{prop:symmetrization}.

\begin{proposition}\label{prop:holonomy-sym}
For every $k$, one has
$$
\widetilde{\mathscr S}_k=\mathscr S_k.
$$
Consequently, $\mathscr S$ preserves reality, bidegree, and strict positivity of positive $(1,1)$-forms and positive $(n-1,n-1)$-forms. Moreover, $\mathscr S\partial=\partial\mathscr S$ and $\mathscr S\overline\partial=\overline\partial\mathscr S$.
\end{proposition}

\begin{proof}
Let $\beta\in\mathcal E_D^{2n-k}$. Since $\beta$ is $D$-parallel,
$P_\gamma^*\beta_x=\beta_p$, and $\beta_p$ is fixed by every
$h\in\mathcal H$. Write
$$
\alpha\wedge\beta=f\,dV_g
$$
for some smooth function $f$ on $M$. Since $P_\gamma^*\beta_x=\beta_p$ and $h^*\beta_p=\beta_p$, while both $P_\gamma$ and $h$ preserve the metric and the orientation, we have
$$
\begin{aligned}
h^*(P_\gamma^*\alpha_x)\wedge\beta_p
&=h^*\bigl(P_\gamma^*\alpha_x\wedge P_\gamma^*\beta_x\bigr)=h^*P_\gamma^*(\alpha_x\wedge\beta_x)
=f(x)dV_{g,p}.
\end{aligned}
$$
Averaging this identity over $\mathcal H$ gives
$$
\mathcal R_\alpha(x)\wedge\beta_p=f(x)dV_{g,p}.
$$
It follows from \eqref{eq:holonomy-average} that
$$
\alpha_p^{\mathrm{av}}\wedge\beta_p
=
\frac{1}{\operatorname{Vol}_g(M)}
\left(\int_M f\,dV_g\right)dV_{g,p}.
$$
Both $\widetilde{\mathscr S}_k(\alpha)$ and $\beta$ are
$D$-parallel, so their wedge product is a $D$-parallel form. Since $DdV_g=0$, it is a constant multiple of $dV_g$, and the
constant is determined by its value at $p$. Therefore
$$
\int_M\widetilde{\mathscr S}_k(\alpha)\wedge\beta
=
\int_M f\,dV_g
=
\int_M\alpha\wedge\beta.
$$
By the uniqueness in Proposition \ref{prop:symmetrization},
$\widetilde{\mathscr S}_k(\alpha)=\mathscr S_k(\alpha)$.

Since $D$ is Hermitian, parallel transport and the holonomy action
are complex linear and unitary. Hence their pull-backs preserve real
forms and the $(p,q)$ decomposition, and the same is true for
$\mathscr S$. Combining this with $\mathscr Sd=d\mathscr S$ and
comparing bidegrees implies
$\mathscr S\partial=\partial\mathscr S$ and
$\mathscr S\overline\partial=\overline\partial\mathscr S$.

For positivity, let $\Omega$ be a strictly positive $(q,q)$-form, where $q=1$ or $q=n-1$. Since $P_\gamma$ and $h$ are complex-linear isometries, their pull-backs preserve the cone of strictly positive $(q,q)$-forms. This cone is convex, and averaging with respect to a probability measure preserves strict positivity. Hence $\mathcal R_\Omega(x)$ is strictly positive for every $x\in M$, and averaging once more over $M$ shows that $\Omega_p^{\mathrm{av}}$ is strictly positive. Since $D$-parallel transport preserves positivity, $\widetilde{\mathscr S}_{2q}(\Omega)$ is strictly positive. Therefore, by $\widetilde{\mathscr S}_{2q}=\mathscr S_{2q}$, the form $\mathscr S\Omega$ is strictly positive.
\end{proof}

We can now symmetrize the two metrics in Conjecture \ref{conj:FV}.

\begin{proposition}\label{prop:sym-metrics}
Suppose that $M$ admits a balanced Hermitian metric and a pluriclosed Hermitian metric. Then there exist a $D$-parallel balanced metric $\omega_B$ and a $D$-parallel pluriclosed metric $\omega_P$.
\end{proposition}

\begin{proof}
Let $\omega_0$ be the fundamental form of a balanced Hermitian metric. Then $\omega_0^{n-1}$ is a strictly positive $(n-1,n-1)$-form and $d\omega_0^{n-1}=0$. Propositions \ref{prop:symmetrization} and \ref{prop:holonomy-sym} show that $\mathscr S(\omega_0^{n-1})$ is $D$-parallel, strictly positive, and $d$-closed.

We now use the standard root construction for positive $(n-1,n-1)$-forms. Michelsohn \cite{Michelsohn} showed that every strictly positive $(n-1,n-1)$-form has a unique strictly positive $(1,1)$-form as its $(n-1)$-st root. Pointwise, this can also be seen from the cofactor map on positive Hermitian matrices: if $G$ is the matrix of a positive $(1,1)$-form, then the matrix associated with its $(n-1)$-st power is, up to the fixed normalization of wedge powers, $(\det G)G^{-1}$. Hence the positive root is unique and depends smoothly on the form. Let $\omega_B$ be characterized by
$$
\omega_B^{n-1}=\mathscr S(\omega_0^{n-1}).
$$
Parallel transport by $D$ preserves the right-hand side and the positive cone. Uniqueness of the positive root therefore gives $D\omega_B=0$, while $d\omega_B^{n-1}=d\mathscr S(\omega_0^{n-1})=0.$ Thus $\omega_B$ is balanced.

Let $\widetilde\omega$ be the fundamental form of a pluriclosed Hermitian metric and set $\omega_P=\mathscr S(\widetilde\omega)$. Proposition \ref{prop:holonomy-sym} gives $\omega_P>0$ and $D\omega_P=0$, and
$$
\partial\overline\partial\omega_P
=
\mathscr S(\partial\overline\partial\widetilde\omega)=0.
$$
Hence $\omega_P$ is pluriclosed.
\end{proof}

\section{The parallel Aeppli slice and the volume functional}

Throughout this section, we keep the notation and assumptions of Section \ref{SYM}. In particular, $D$ is Hermitian and $DT^D=0$. We now work in the finite-dimensional space of parallel real $(1,1)$-forms
$$
\mathcal E_D^{1,1}(\mathbb R)
=
\{\alpha\in\Omega^{1,1}(M):\alpha\text{ is real and }D\alpha=0\}.
$$
Let $\omega_B$ and $\omega_P$ be the metrics constructed in Proposition \ref{prop:sym-metrics}. Motivated by the invariant Aeppli directions and the finite-dimensional variational argument in \cite{Kwong}, we define the space of $D$-parallel Aeppli directions by
\begin{equation}\label{eq:A-space}
\mathcal A_D=\{(d\beta)^{1,1}:\beta\in\mathcal E_D^1\ \text{is real}\}.
\end{equation}
By Lemma \ref{lem:d-parallel} and the fact that $D$ preserves bidegree, $\mathcal A_D\subset\mathcal E_D^{1,1}(\mathbb R)$.

\begin{lemma}\label{lem:A-ddbar}
Every $\alpha\in\mathcal A_D$ is $\partial\overline\partial$-closed. Consequently, every positive definite form in $\omega_P+\mathcal A_D$ is pluriclosed.
\end{lemma}

\begin{proof}
Let $\alpha=(d\beta)^{1,1}\in\mathcal A_D$, where
$\beta\in\mathcal E_D^1$ is real. Writing
$\beta=\beta^{1,0}+\beta^{0,1}$, we have
$$
(d\beta)^{1,1}=\overline\partial\beta^{1,0}+\partial\beta^{0,1}.
$$
Hence, using $\partial^2=\overline\partial^{\,2}=0$ and
$\partial\overline\partial=-\overline\partial\partial$, one gets
$$
\partial\overline\partial\alpha
=
\partial\overline\partial
\big(\overline\partial\beta^{1,0}+\partial\beta^{0,1}\big)=0.
$$
Since $\partial\overline\partial\omega_P=0$, every positive form in
$\omega_P+\mathcal A_D$ is pluriclosed.
\end{proof}

Let
$$
\mathscr P
=
\{\omega\in\omega_P+\mathcal A_D:\omega>0\}.
$$
Thus $\mathscr P$ is nonempty because $\omega_P\in\mathscr P$. On the finite-dimensional space of parallel $(1,1)$-forms, consider the volume functional
$$
\mathcal V(\omega)=\frac1{n!}\int_M\omega^n.
$$
For a $D$-parallel real $(1,1)$-form $\omega$ and $\alpha\in\mathcal E_D^{1,1}(\mathbb R)$, write
$$
(d\mathcal V)_\omega(\alpha)
=
\left.\frac{d}{dt}\right|_{t=0}\mathcal V(\omega+t\alpha).
$$

\begin{proposition}\label{prop:critical-balanced}
Let $\omega$ be a positive definite $D$-parallel real $(1,1)$-form. Then $\omega$ is balanced if and only if $(d\mathcal V)_\omega(\alpha)=0$ for every $\alpha\in\mathcal A_D$.
\end{proposition}

\begin{proof}
Take $\alpha=(d\beta)^{1,1}\in\mathcal A_D$. We have
\begin{align*}
(d\mathcal V)_\omega(\alpha)
&=\frac1{(n-1)!}\int_M(d\beta)^{1,1}\wedge\omega^{n-1}
=\frac1{(n-1)!}\int_M\beta\wedge d\omega^{n-1}.
\end{align*}
If $\omega$ is balanced, the last integral is zero for all $\beta$, so $\omega$ is critical along $\mathcal A_D$.

Conversely, suppose $(d\mathcal V)_\omega(\alpha)=0$ for every $\alpha\in\mathcal A_D$. Since $\omega$ is $D$-parallel, so is $\omega^{n-1}$, and Lemma \ref{lem:d-parallel} gives
$$
d\omega^{n-1}\in\mathcal E_D^{2n-1}.
$$
The formula above gives
$$
\int_M\beta\wedge d\omega^{n-1}=0
$$
for every real $D$-parallel $1$-form $\beta$. By Lemma \ref{lem:integration-pairing}, the pairing between $\mathcal E_D^1$ and $\mathcal E_D^{2n-1}$ is nondegenerate, hence $d\omega^{n-1}=0$ and $\omega$ is balanced.
\end{proof}

We next apply Proposition \ref{prop:critical-balanced} to the balanced form $\omega_B$.

\begin{lemma}\label{lem:traceless}
For every $\alpha\in\mathcal A_D$, one has $\operatorname{tr}_{\omega_B}\alpha=0$ at every point of $M$.
\end{lemma}

\begin{proof}
Since $\omega_B$ is balanced, Proposition \ref{prop:critical-balanced} gives $(d\mathcal V)_{\omega_B}(\alpha)=0$. By the standard trace identity for $(1,1)$-forms,
$$
\alpha\wedge\omega_B^{n-1}
=\frac1n\operatorname{tr}_{\omega_B}(\alpha)\,\omega_B^n.
$$
Consequently,
$$
0=\int_M\alpha\wedge\omega_B^{n-1}
=\frac1n\int_M\operatorname{tr}_{\omega_B}(\alpha)\,\omega_B^n.
$$
It remains to pass from the vanishing of the integral to a pointwise
statement. Since both $\alpha$ and $\omega_B$ are $D$-parallel, their contraction $\operatorname{tr}_{\omega_B}\alpha$ is $D$-parallel. Hence $D(\operatorname{tr}_{\omega_B}\alpha)=0$. Thus $\operatorname{tr}_{\omega_B}\alpha$ is a constant on the connected manifold $M$. Since $\int_M\omega_B^n>0$, the previous integral identity forces this constant to be zero.
\end{proof}

Since both $\omega_B$ and $\omega_P$ are $D$-parallel, the function $\operatorname{tr}_{\omega_B}\omega_P$ is constant on the connected manifold $M$. It follows from Lemma \ref{lem:traceless} that for every $\alpha \in \mathcal A_D$, $\omega=\omega_P+\alpha\in\mathscr P$ has the same $\omega_B$-trace, namely
\begin{equation}\label{eq:fixed-trace}
\operatorname{tr}_{\omega_B}\omega=\operatorname{tr}_{\omega_B}\omega_P=:C>0.
\end{equation}
For any $D$-parallel real $(1,1)$-form $\omega$, let $A_\omega$ be the $\omega_B$-Hermitian endomorphism determined by $\omega(X,\overline Y)=\omega_B(A_\omega X,\overline Y)$ for $X,Y\in T^{1,0}M$. Since both forms are $D$-parallel, $DA_\omega=0$. Thus the eigenvalues, the trace, and the determinant of $A_\omega$ are constant on the connected manifold. Pointwise linear algebra gives
\begin{equation}\label{eq:volume-det}
\omega^n=\det(A_\omega)\omega_B^n,\qquad \mathcal V(\omega)=\det(A_\omega)\mathcal V(\omega_B).
\end{equation}

\begin{lemma}\label{lem:compact-slice}
The set $\mathscr P$ is bounded in the affine space $\omega_P+\mathcal A_D$. Its relative closure is
$$
\overline{\mathscr P}
=
\{\omega\in\omega_P+\mathcal A_D:\omega\geq0\},
$$
and is compact.
\end{lemma}

\begin{proof}
Fix $p\in M$. Evaluation at $p$ is injective on the space of $D$-parallel forms, so we identify such forms with their values at $p$. For $\omega\in\mathscr P$, the endomorphism $A_\omega$ is positive definite and Hermitian. Let $\lambda_1,\ldots,\lambda_n>0$ be its eigenvalues. From \eqref{eq:fixed-trace},
$$
\lambda_1+\cdots+\lambda_n
=
\operatorname{tr}(A_\omega)
=
C.
$$
Thus $0<\lambda_i\leq C$ for all $i$, and the operator norms of $A_\omega$ are uniformly bounded. Since the map $\omega\mapsto A_\omega$ is linear after evaluation at $p$, $\mathscr P$ is bounded in the finite-dimensional affine space $\omega_P+\mathcal A_D$.

Any limit of positive definite Hermitian forms is semipositive, so
$$
\overline{\mathscr P}
\subseteq
\{\omega\in\omega_P+\mathcal A_D:\omega\geq0\}.
$$
Conversely, let $\omega=\omega_P+\alpha$ in the right-hand side. For $0<t\leq1$, set
$$
\omega_t=(1-t)\omega+t\omega_P
=
\omega_P+(1-t)\alpha.
$$
Because $\omega$ is semipositive and $\omega_P$ is positive definite, $\omega_t$ is positive definite and belongs to $\mathscr P$. Since $\omega_t\to\omega$ as $t\to0$, the reverse inclusion follows. The relative closure is therefore closed and bounded in a finite-dimensional affine space, and hence compact.
\end{proof}

The compactness established above allows us to maximize the volume functional on $\overline{\mathscr P}$.

\begin{proposition}\label{prop:common-metric}
The positive slice $\mathscr P$ contains a $D$-parallel Hermitian metric $\omega_*$ which is simultaneously balanced and pluriclosed.
\end{proposition}

\begin{proof}
By Lemma \ref{lem:compact-slice}, $\overline{\mathscr P}$ is compact, so the continuous function $\mathcal V$ attains a maximum there. If $\omega$ lies on the relative boundary of $\mathscr P$, then $\omega$ is semipositive but not positive definite. Hence $A_\omega$ is semipositive and singular, so $\det(A_\omega)=0$. By \eqref{eq:volume-det}, $\mathcal V(\omega)=0$. On the other hand, $\omega_P\in\mathscr P$ and $\mathcal V(\omega_P)>0$. Therefore every maximum point of $\mathcal V$ on $\overline{\mathscr P}$ lies in the relative interior $\mathscr P$.

Choose a maximizer $\omega_*\in\mathscr P$. For every $\alpha\in\mathcal A_D$, positivity is preserved along $\omega_*+t\alpha$ for sufficiently small $|t|$, and hence
$$
0
=
\left.\frac{d}{dt}\right|_{t=0}\mathcal V(\omega_*+t\alpha)
=
(d\mathcal V)_{\omega_*}(\alpha).
$$
Proposition \ref{prop:critical-balanced} therefore shows that $\omega_*$ is balanced. By Lemma \ref{lem:A-ddbar}, every metric in $\mathscr P$ is pluriclosed, and hence $\partial\overline\partial\omega_*=0$. Finally, $\omega_*\in\omega_P+\mathcal A_D\subset\mathcal E_D^{1,1}(\mathbb R)$, so $D\omega_*=0$.
\end{proof}

\section{Proofs of the main results}

\begin{proof}[Proof of Theorem \ref{thm:parallel-main}]
If $n=1$, the given Hermitian metric $g$ is K\"ahler and is already $D$-parallel. We therefore assume $n\geq2$. Proposition \ref{prop:sym-metrics} produces $D$-parallel metrics $\omega_B$ and $\omega_P$ which are balanced and pluriclosed, respectively. Proposition \ref{prop:common-metric} then gives a $D$-parallel metric $\omega_*$ which is both balanced and pluriclosed. Lemma \ref{lem:balanced-pluriclosed-kahler} implies $d\omega_*=0$. Thus $\omega_*$ is a $D$-parallel K\"ahler metric.
\end{proof}

\begin{proof}[Proof of Corollary \ref{thm:main}]
Let $(M^n,g)$ be a compact BTP manifold. Taking $D=\nabla^b$ in Theorem \ref{thm:parallel-main} shows that whenever $M$ admits both a balanced Hermitian metric and a pluriclosed Hermitian metric, it admits a K\"ahler metric. This proves the first assertion.

Assume in addition that the BTP metric $g$ is non-balanced. If the underlying complex manifold admitted both a balanced metric and a pluriclosed metric, Theorem \ref{thm:parallel-main} would again produce a K\"ahler metric. This contradicts Lemma \ref{lem:ZZ-obstruction}. Hence the two types of metrics cannot coexist.
\end{proof}

\vspace{0.3cm}

\noindent\textbf{Generative AI disclosure.}
During the development of this work, the authors used ChatGPT
5.6 Sol to assist with exploratory computations, possible proof directions, and editorial polishing.
The authors checked the mathematical proofs and computations, and take full responsibility for
the contents of the paper.

\vspace{0.3cm}

\noindent\textbf{Declaration on competing interests.}
All authors declare that there are no competing interests for this paper.

\end{document}